\documentclass[
submission
]{dmtcs-episciences}

\usepackage[utf8]{inputenc}
\usepackage{subfigure}

\usepackage{color}

\usepackage{calc}
\usepackage{array}
\usepackage{subfigure}
\usepackage[indent,bf]{caption}
\usepackage{rotating}
\usepackage{setspace}
\usepackage{longtable}
\usepackage{mdframed}
\usepackage{rotating}
\usepackage{hyperref}
\usepackage{amssymb, amsmath, amsfonts, amsthm, epsfig}
\usepackage{epsfig,latexsym}
\usepackage{multicol, paralist}
\usepackage{pgf,tikz}
\usetikzlibrary{arrows,snakes,backgrounds, automata}
\usepackage{verbatim}
\usepackage{url}
\usepackage{fancyhdr}
\usepackage{marginnote}

\usepackage{hyperref}

\newfont{\blb}{msbm10 scaled\magstep1}
\newtheorem*{HConj}{Hendry's Conjecture}

\newcommand{\Z}{\mathbb{Z}}

\newtheorem{lem}{Lemma}[section]

\newtheorem{cor}[lem]{Corollary}

\newtheorem{thm}[lem]{Theorem}
\newtheorem{ques}[lem]{Question}

\newtheorem{prop}[lem]{Proposition}
\newtheorem{prob}[lem]{Problem}

\author{Manuel Lafond\affiliationmark{1}
  \and Ben Seamone\affiliationmark{2,3}\thanks{Funded by Fonds de recherche du Qu\'ebec – Nature et technologies grant no.~185200 and Natural Sciences and Engineering Research Council of Canada grant no.~6790-421798-2012}
  \and Rezvan Sherkati\affiliationmark{4}}

\title[Further results on Hendry's Conjecture]{Further results on Hendry's Conjecture}

\affiliation{
  D\'epartement d'informatique, Universit\'e de Sherbrooke, Canada\\
  Department d'informatique et de recherche op\'erationnelle, Universit\'e de Montr\'eal, Canada\\
  Department of Mathematics, Dawson College, Montreal, Canada \\
  Department of Electrical and Computer Engineering, McGill University, Montr\'eal, Canada}
\keywords{Chordal graphs, strongly chordal graphs, Hamiltonian graphs, cycle extendable graphs}
\received{2020-08-06}
\revised{2022-07-11}
\accepted{2022-07-13}
\begin{document}
\publicationdetails{24}{2022}{2}{1}{6700}

\maketitle

\begin{abstract}
Recently, a conjecture due to Hendry which stated that every Hamiltonian chordal graph is cycle extendable was disproved.  Here we further explore the conjecture, showing that it fails to hold even when a number of extra conditions are imposed.  In particular, we show that Hendry's Conjecture fails for strongly chordal graphs, graphs with high connectivity, and if one relaxes the definition of ``cycle extendable'' considerably.  We also consider the original conjecture from a sub-tree intersection model point of view, showing that a result of Abuieda et al.~is nearly best possible.
\end{abstract}

\section{Introduction}

All graphs considered here are simple, finite, connected, and undirected.  A {\em Hamiltonian cycle} is a cycle of a graph that contains every vertex; a graph that contains a Hamiltonian cycle is called {\em Hamiltonian}.  
 A graph $G$ on $n$ vertices is {\em pancyclic} if $G$ contains a cycle of length $m$ for every integer $3 \leq m \leq n$.  
 A graph $G$ is {\em cycle extendable} if, for every non-Hamiltonian cycle $C$, there exists a cycle $C'$ such that $V(C) \subset V(C')$ and $|V(C')| = |V(C)| + 1$ (we say that $C$ {\em extends} to $C'$).
If, in addition, every vertex of $G$ is contained in a triangle, then $G$ is {\em fully cycle extendable}.

A graph $H$ is an \emph{induced subgraph} of $G$ if $H$ can be obtained from $G$ by deleting vertices and all edges incident to these vertices.  The remaining vertices are said to induce $H$.  If no induced subgraph of $G$ is isomorphic to $H$, $G$ is said to be \textit{$H$-free}.  A graph is {\em chordal} if it is $C_k$-free for every integer $k \geq 4$; that is, every cycle of length $4$ or greater has a chord.
A graph is {\em strongly chordal} if it is chordal and every even cycle of length at least $6$ has a chord that connects vertices at an odd distance from one another along the cycle.

The results in this paper are motivated by the following conjecture:

\begin{HConj}\cite{H90}
If $G$ is a Hamiltonian chordal graph, then $G$ is fully cycle extendable.
\end{HConj}

In \cite{LS14}, the first two authors answered Hendry's Conjecture in the negative:

\begin{thm}
For any $n \geq 15$, there exists a Hamiltonian chordal graph $G$ on $n$ vertices which is not fully cycle extendable.
\end{thm}

In this paper, we improve on the counterexample construction given in \cite{LS14} to show that counterexamples to Hendry's Conjecture exist even in highly restrictive settings.  For instance, in \cite{LS14}, it was asked whether or not Hendry's Conjecture holds for either strongly chordal graphs.  This question is answered in the negative in Section \ref{strongex}.  In Section \ref{modified}, we examine ways in which we can modify the construction given in Section \ref{strongex} to obtain examples which satisfy even stronger conditions in terms of forbidden induced paths (improving on a result from \cite{LS14}), connectivity (answering a question from \cite{LS14}, $S$-extendability (answering a conjecture from \cite{A14}), and underlying tree structure (showing that results in \cite{ABS13} are almost best possible).  Note that similar results to some of those presented in Sections 2 and 3 were independently obtained by Rong et al.~\cite{RLWY20}.
Finally, we propose an extremal problem in Section \ref{conclusion} related to our counterexample constructions.

\section{A new family of counterexamples to Hendry's Conjecture}\label{strongex}

In this section, we describe a family of strongly chordal graphs for which Hendry's Conjecture fails to hold; that is, the graphs are Hamiltonian and chordal, but not cycle extendable.

The {\em join} of two disjoint graphs $G$ and $H$, denoted $G \vee H$, is the graph with vertex set $V(G) \cup V(H)$ and edge set $E(G) \cup E(H) \cup \{uv \,:\, u \in V(G), v \in V(H)\}$.  For the remainder of this note, we let $G_k := K_k \vee P_{2k+1}$, where the vertices of the complete graph $K_k$ are denoted $\{x_1, x_2, \ldots, x_k\}$ and the path $P_{2k+1}$ has vertices (in order) $\{u_1, u_2, \ldots u_k, z, v_k, \ldots, v_2, v_1\}$.  Let $A_k  = \{x_iu_i \,:\, 1 \leq i \leq k\} \cup \{x_iv_i \,:\, 1 \leq i \leq k-1\} \subset E(G_k)$; we call an edge $e \in A_k$ a {\em heavy edge} of $G_k$.  Figure \ref{base} depicts $G_k$ with the heavy edges $A_k$ in bold. 

\begin{figure}[htb!]
\begin{center}
\scalebox{0.7}{
\begin{tikzpicture}
\clip(-2.5,-1.5) rectangle (14.5, 9);

\draw (-2,0) -- (-1,0);
\draw (1,0) -- (11,0);
\draw (13,0) -- (14,0);

\draw (6,0) -- (6,5);
\draw (6,7) -- (6,8);

\draw (6,2) -- (-2,0);
\draw (6,2) -- (2,0);
\draw (6,2) -- (4,0);
\draw (6,2) -- (6,0);
\draw (6,2) -- (8,0);
\draw (6,2) -- (10,0);
\draw (6,2) -- (14,0);

\draw (6,4) to [out=200, in=45] (-2,0);
\draw (6,4) to [out=210, in=60] (2,0);
\draw (6,4) to [out=220, in=75] (4,0);
\draw (6,4) to [out=300, in=60] (6,0);
\draw (6,4) to [out=320, in=105] (8,0);
\draw (6,4) to [out=330, in=120] (10,0);
\draw (6,4) to [out=340, in=135] (14,0);

\draw (6,8) to [out=300, in=60] (6,2);

\draw (6,8) to [out=200, in=90] (-2,0);
\draw (6,8) to [out=210, in=90] (2,0);
\draw (6,8) to [out=220, in=90] (4,0);
\draw (6,8) to [out=310, in=50] (6,0);
\draw (6,8) to [out=320, in=90] (8,0);
\draw (6,8) to [out=330, in=90] (10,0);
\draw (6,8) to [out=340, in=90] (14,0);

\draw [line width=6pt]  (6,2) -- (4,0);
\draw [line width=6pt]  (6,4) to [out=210, in=60] (2,0);
\draw [line width=6pt]  (6,4) to [out=330, in=120] (10,0);
\draw [line width=6pt]  (6,8) to [out=200, in=90] (-2,0);
\draw [line width=6pt]  (6,8) to [out=340, in=90] (14,0);

\begin{scriptsize}

\fill [color=black] (-2,0) circle (6pt);
\fill [color=black] (2,0) circle (6pt);
\fill [color=black] (4,0) circle (6pt);
\fill [color=black] (6,0) circle (6pt);
\fill [color=black] (8,0) circle (6pt);
\fill [color=black] (10,0) circle (6pt);
\fill [color=black] (14,0) circle (6pt);

\fill [color=black] (6,2) circle (6pt);
\fill [color=black] (6,4) circle (6pt);
\fill [color=black] (6,8) circle (6pt);

\fill [color=black] (-0.5,0) circle (1pt);
\fill [color=black] (0,0) circle (1pt);
\fill [color=black] (0.5,0) circle (1pt);

\fill [color=black] (11.5,0) circle (1pt);
\fill [color=black] (12,0) circle (1pt);
\fill [color=black] (12.5,0) circle (1pt);

\fill [color=black] (6,5.5) circle (1pt);
\fill [color=black] (6,6) circle (1pt);
\fill [color=black] (6,6.5) circle (1pt);

\end{scriptsize}

\draw (-2, -0.5) node[anchor=center] {\Large $u_1$};
\draw (2, -0.5) node[anchor=center] {\Large $u_{k-1}$};
\draw (4, -0.5) node[anchor=center] {\Large $u_k$};
\draw (6, -0.5) node[anchor=center] {\Large $z$};
\draw (8, -0.5) node[anchor=center] {\Large $v_k$};
\draw (10, -0.5) node[anchor=center] {\Large $v_{k-1}$};
\draw (14, -0.5) node[anchor=center] {\Large $v_1$};

\draw (5.6, 2.4) node[anchor=center] {\Large $x_k$};
\draw (5.5, 4.4) node[anchor=center] {\Large $x_{k-1}$};
\draw (5.6, 8.4) node[anchor=center] {\Large $x_1$};

\end{tikzpicture}
}
\caption{The base graph $G_k$}\label{base}
\end{center}
\end{figure}
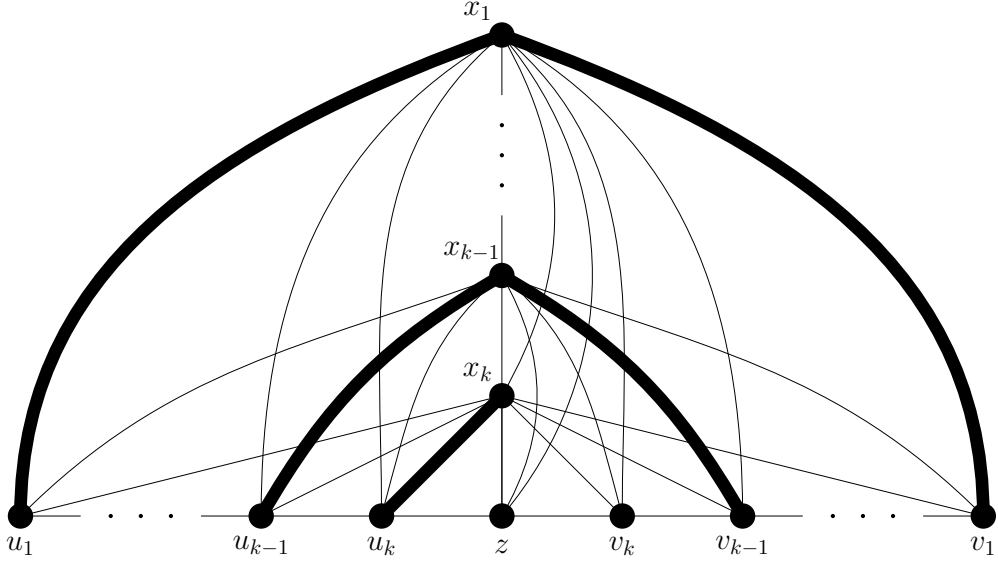

Rather than work directly with the definition of strongly chordal graphs, we use an equivalent characterization, given by 
Farber in \cite{F83}.
Let $G$ be a graph.  A vertex $v \in V(G)$ is called {\em simple} if for all $x,y \in N[v]$ either $N[x] \subseteq N[y]$ or $N[y] \subseteq N[x]$.  In other words, $x$ is simple if the closed neighbourhoods of its neighbours can be ordered by inclusion.
A {\em simple elimination ordering} of $G$ is a vertex ordering $v_1 \prec v_2 \prec \ldots \prec v_n$ such that $v_i$ is simple in $G[\{v_i, \ldots, v_n\}]$ for every $1 \le i \le n$.

\begin{thm}[Farber \cite{F83}]\label{Farber}
A graph $G$ is strongly chordal if and only if $G$ admits a simple elimination ordering.
\end{thm}

It is easy to check that the following ordering of $E(G_k)$ is a simple elimination ordering:
$$u_1 \prec u_2 \prec \cdots \prec u_k \prec v_1 \prec v_2 \prec \cdots \prec v_k \prec x_1 \prec x_2 \prec \cdots \prec x_k \prec z,$$
and so we have the following:

\begin{lem}\label{strong}
For any $k \geq 1$, $G_k$ is strongly chordal.
\end{lem}

A cycle $C$ is a {\em heavy cycle} of $G_k$ if $A_k \subset E(C)$.

\begin{lem}\label{ham}
For any $k \geq 1$, $G_k$ has a heavy Hamiltonian cycle $C$.
\end{lem}

\begin{proof}
If $k$ is even, then there is a heavy Hamiltonian cycle with edge set $A_k \cup \{u_1x_k, u_2u_3, \ldots u_{k-2}u_{k-1}, u_kz, zv_k, v_kv_{k-1}, \ldots, v_2v_1\}$.  If $k$ is odd, then there is a heavy Hamiltonian cycle with edge set $A_k \cup \{u_1u_2, \ldots, u_{k-2}u_{k-1}, u_kz, zv_k, v_kv_{k-1}, \ldots, v_3v_2, v_1x_k\}$. 
\end{proof}

\begin{lem}\label{long}
For any $k \geq 2$, $G_k$ has a heavy cycle $C_1$ such that $V(C_1) = V(G_k) \setminus \{v_k,z\}$.
\end{lem}

\begin{proof}
If $k$ is even, then there is a heavy cycle with edge set $A_k \cup \{u_1u_2, \ldots, u_{k-1}u_{k}, v_{k-1}v_{k-2}, \ldots, v_3v_2, v_1x_k\}$.  If $k$ is odd, then there is a heavy cycle with edge set $A_k \cup \{u_1x_k, u_2u_3, \ldots, u_{k-1}u_{k}, v_{k-1}v_{k-2}, \ldots, v_2v_1\}$.
\end{proof}

\begin{lem}\label{noextend}
For any $k \geq 3$, $G_k$ does not contain a heavy cycle whose vertex set is either $V(G) \setminus \{z\}$ or $V(G) \setminus \{v_k\}$.
\end{lem}

\begin{proof}
Suppose, to the contrary, such a cycle does exist.  It clearly cannot contain any edge incident to a non-heavy edge in $\{x_1, \ldots, x_{k-1}\}$, nor can it contain the edge $zv_k$.  This immediately implies that no heavy cycle exists with vertex set $V(G) \setminus \{v_k\}$, since the only available edges incident to $z$ are $zu_k$ and $zx_k$.  This, in turn, implies that both $u_{k-1}u_{k}$ and $v_{k-1}v_{k}$ must be in such a cycle.  Hence our cycle must be exactly $u_{k-1}x_{k-1}v_{k-1}v_kx_ku_ku_{k-1}$, contradicting the fact that $k \geq 3$.
\end{proof}

We now define the family of graphs $\mathcal{H}_k$ as those which can be obtained from $G_k$ by pasting a distinct clique onto each edge in $A_k$.

\begin{lem}\label{bigstrong}
For each $k \geq 3$, every graph in $\mathcal{H}_k$ is strongly chordal.
\end{lem}

\begin{proof}
By Theorem \ref{Farber}, we need only show that $H$ admits a simple elimination ordering.  Let $G_k$ be the graph from which $\mathcal{H}_k$ is obtained.  For each $i = 1, 2, \ldots, k$, $N_{G_k}(u_i) \subseteq N_{G_k}(x_i)$ and $N_{G_k}(v_i) \subseteq N_{G_k}(x_i)$.  It follows that we can obtain a simple elimination ordering of $H$ as follows: first, for each complete graph pasted onto a heavy edge of $G_k$, take in any order all of its vertices except those in the edge onto which it is pasted in $H$, then finish by taking a simple elimination ordering of $G_k$ guaranteed by Lemma \ref{strong}. 
\end{proof}

\begin{thm}\label{result}
If $H \in \mathcal{H}_k$ for some $k \geq 3$, then $H$ is Hamiltonian but not cycle extendable.
\end{thm}

\begin{proof}
Let $G_k$ be the graph from which $H$ is obtained.
Lemma \ref{ham} states that $G_k$ contains a heavy Hamiltonian cycle.  By replacing each heavy edge in $G_k$ with a Hamiltonian path in its corresponding clique connecting its ends, we see that $H$ has a Hamiltonian cycle.  Similarly, Lemma \ref{long} guarantees that $H$ has a cycle of length $|V(H)| - 2$ that does not contain $z$ or $v_k$.  Call this cycle $C$.  If $C$ were to extend in $H$, then this cycle would correspond to a heavy cycle in $G_k$ that avoids only $z$ or $v_k$.  Since this cannot happen by Lemma \ref{noextend}, $H$ is not cycle extendable.
\end{proof}

Since Lemma \ref{bigstrong} states that the graphs in the proof above are strongly chordal, we have the following:

\begin{cor}
For any $n \geq 15$, there exists a Hamiltonian strongly chordal graph $G$ on $n$ vertices which is not fully cycle extendable.
\end{cor}

We noted in \cite{LS14} that bull-free chordal graphs form a subclass of strongly chordal graphs.  Let $H \in H_k$, $w$ be a vertex in the clique pasted onto $u_1x_1$, and $y$ be a vertex in the clique pasted onto 
$v_2x_2$.  Since $\{u_1, x_2, x_3, w,y\}$ 
induce a bull, 
we see that our counterexamples are not bull-free.

\begin{ques}
Are Hamiltonian bull-free chordal graphs fully cycle extendable?
\end{ques}




\section{Counterexamples satisfying stronger conditions}\label{modified}

In this section, we show how the family of graphs $\mathcal{H}_k$ defined in Section \ref{strongex} can be modified in order to answer even more open problems on extending cycles in chordal graphs.  In particular, we show that there exist counterexamples to Hendry's Conjecture which are $P_9$-free (an improvement on a result from \cite{LS14}) and counterexamples which have connectivity $k$ for every integer $k \geq 2$ (answering a question from \cite{LS14}).  We also show that Hendry's Conjecture fails even if the extendability condition is relaxed to only require that each cycle be extendable by some length in any predefined finite subset of $\Z^+$.  Finally, we examine our counterexamples from the point of view of intersection graphs.

\subsection{$P_k$-free graphs}\label{nopath}

For a graph $H \in \mathcal{H_k}$, let $H^+$ be the graph obtained by adding the edge $u_1u_3$ to the base graph $G_k$.  It is easy to check that $H^+$ is still strongly chordal and Hamiltonian.  Furthermore, the proof of Lemma \ref{noextend} applies to $G_k + u_1u_3$, and so $H^+$ is not cycle extendable.

\begin{thm}\label{P9}
For any $n \geq 15$, there exists a $P_9$-free counterexample to Hendry's Conjecture on $n$ vertices.
\end{thm}

\begin{proof}
Let $H \in \mathcal{H}_3$ with $|V(H)| = n \geq 15$, and let $H^+$ be as defined above.  Note that the longest induced path in $G_3 + u_1u_3$ has $6$ vertices ($u_1u_3zv_3v_2v_1$).  Since any longest induced path in $H^+$ contains at most two vertices not in $V(G_3)$, $H^+$ must be $P_9$-free.
\end{proof}

In \cite{LS14}, it was asked to determine the smallest positive integer $k$ such that every $P_k$-free Hamiltonian chordal graph is cycle extendable.  Combined with the results of \cite{LS14}, Theorem \ref{P9} shows that the smallest such $k$ satisfies $5 \leq k \leq 8$.

\subsection{Connectivity}\label{conn}


Informally, the following construction involves taking $G_k$ and replacing each $u_i$ and $v_i$ with a clique of order $k-1$ in a particular way.  Heavy edges transform into ``heavy cliques'' of order $k$, to which we join independent sets of order $k-1$.

More precisely, let $F_1,F_2,\ldots , F_{k-1},F'_1,F'_2,\ldots , F'_{k-2}$  be isomorphic to $K_{k-1}$.  We denote their respective vertex sets by $V(
F_i) = \{ v_{i,1}, v_{i,2},\ldots , v_{i,k-1}\}$ and $V(F'_i) = \{ v'_{i,1}, v'_{i,2},\ldots , v'_{i,k-1}\}$.  Let $Q_{k-1}$ be the graph with vertices and edges as follows:
 \begin{align*}
 V(Q_{k-1})
 = &V(F_1) \cup V(F_2) \cup \ldots \cup V(F_{k-1}) \cup  V(F'_1) \cup V(F'_2) \cup \ldots \cup V(F'_{k-2}) \cup \{z,v_k \} \\ E(Q_{k-1}) = &\bigcup_{i=1}^{k-1} E(F_i) \cup \bigcup_{i=1}^{k-2}E(F'_i) \cup \{ v_{1,k-1}v_{2,1}, v_{2,k-1}v_{3,1},\ldots,v_{k-2,k-1}v_{k-1,1}\} \cup  \\ &\{ v_{k-1,k-1}z, zv_k, v_kv'_{k-2,1} \} \cup \{ v'_{k-2,k-1}v'_{k-3,1}, v'_{k-3,k-1}v'_{k-4,1},\ldots,v'_{2,k-1}v'_{1,1} \}
 \end{align*}
 Let $Z_{k-1}$ be a complete graph on vertices $\{x_1, \ldots, x_{k-1}\}$ and let $R_{k-1} = Z_{k-1} \vee Q_{k-1}$.  Note, at this point, that contracting each $F_i$ and $F'_i$ to single vertices yields $G_{k-1}$.
 
 For $1 \leq i \leq k-1$, let $T_i$ denote an independent set of vertices $\{t_{i,1},t_{i,2},\ldots,t_{i,k-1}\}$.  Similarly, for $1 \leq i \leq k-2$, let $T'_i$ denote an independent set of vertices $\{t'_{i,1},t'_{i,2},\ldots,t_{i,k-1}\}$.  Let $S_{k-1}$ be the graph obtained from $R_{k-1}$ by making each $T_i$ (respectively, $T'_i$) complete to $F_i \cup x_i$ (resp., $F'_i \cup x_i$).  To extend the point made in the previous paragraph, note that contracting each $F_i, F'_i, T_i$, and $T'_i$ to single vertices yields a graph in $\mathcal{H}_{k-1}$ (where the clique pasted to each heavy edge is a triangle).
 
 \begin{prop}
 The graph $S_{k-1}$ is chordal, Hamiltonian, and $k$-connected.
 \end{prop}
 
 \begin{proof}
 It is straightforward to confirm that $Q_{k-1}$ is chordal (as is $Z_{k-1}$), and thus $R_{k-1}$ is the join of two chordal graphs and is thus chordal itself.  Since $S_{k-1}$ is obtained from $R_{k-1}$ by clique pasting, it is also chordal.  We first note that there is a Hamiltonian cycle $C'$ in $R_{k-1}$ which is obtained from a heavy Hamiltonian cycle $C$ in $G_{k-1}$ in a natural way -- each $u_i$ and $v_j$ in $C$ is replaced with the vertices of $F_i$ and $F'_j$, respectively, in an appropriate order and each heavy edge is replaced with an appropriate edge from $Z_{k-1}$ to the corresponding copy of $F_i$ or $F_j'$ (we still call such edges ``heavy'').  Now, we obtain a Hamiltonian cycle $C''$ in $S_{k-1}$ -- for each edge in $E(F_i) \cap E(C')$ replace it with a two-edge path whose centre is a vertex from $T_i$, for each edge in $E(F'_j) \cap E(C')$ replace it with a two-edge path whose centre is a vertex from $T'_j$, and replace each heavy edge with a two-edge path containing the only remaining vertex from the set $T_i$ or $T'_j$ which is complete to that edge.  It is straightforward case analysis to confirm that there are $k$ internally disjoint paths between any two vertices, which we leave to the reader.
 \end{proof}
 
\begin{lem}\label{connectedblowup}
Let $k \geq 3$.  For $y \in \{z,v_k\}$, the graph $S_{k-1}$ has no cycle $C$ such that $V(C) = V(S) \setminus \{y\}$ but has a cycle $C'$ such that $V(C) = V(S) \setminus \{z,v_k\}$.
\end{lem}

\begin{proof}
It is straightforward to obtain $C'$ be extending the cycle in $H$ which avoids $\{v_k, z\}$ given by Lemma \ref{long}.  To see that this is possible, note that if $x$ is a degree two vertex in $H$ which is part of a clique pasted to a heavy edge, then the two edges in $C'$ incident to $x$ can be replaced in $C$ by a path tracing all vertices in the appropriate $G_i \cup T_i \cup \{x_i\}$ or $G'_i \cup T'_i \cup \{x'_i\}$.  An edge of the form $u_iu_{i+1}$ is replaced by the edge $v_{i,k-1}v_{i+1,1}$ (similar for $v_iv_{i+1}$ in $H$), and any edge from some $u_i$ or $v_i$ to some $x_j$ can replaced with an edge from $x_j$ to whichever vertex in $G_i$ or $G'_i$ is the end of the path tracing the vertices of $T_i$ noted above.
Now, assume that $S$ has a cycle $C$ such that $V(C) = V(S) \setminus \{y\}$ for either choice of $y \in \{z,v_k\}$.  Consider the portion of $C$ induced by $G_i \cup T_i \cup \{x_i\}$ for some $i$.  There are $2k-2$ edges of $C$ incident with the independent set $T_i$, and all must have their other ends in $G_i \cup \{x_i\}$.  This means that $C$ has precisely two edges each having one end in $G_i \cup \{x_i\}$ and other ends outside of $G_i \cup T_i \cup \{x_i\}$, and that the portion of $C$ induced by $G_i \cup T_i \cup \{x_i\}$ is a path $P_i$.  A similar argument holds for $G'_i \cup T'_i \cup \{x_i\}$, yielding a path through those vertices which we denote $P'_i$.  Note further that $P_i$ cannot begin and end at vertices in $G_i$, as this would leave $x_i$ unavailable for use in $P'_i$ (and conversely).  We construct a corresponding cycle $C^*$ in $G_{k-1}$ by replacing each $P_i$ with the heavy edge $u_ix_i$ and each $P'_i$ with the heavy edge $v_ix_i$.  By the arguments in Section $2$ it is not possible to have a cycle in $G_{k-1}$ which uses every heavy edge and misses only a vertex in $\{v_k,z\}$, contradicting our initial claim that $C$ exists.
\end{proof}

As consequence of Lemma \ref{connectedblowup}, we obtain the following:

\begin{thm}
For any $k \geq 2$, there exists a counterexample to Hendry's Conjecture with connectivity $k$.
\end{thm}

\subsection{$S$-extendability}\label{length}

Let $S \subset \Z^+$ be finite.  A cycle $C$ in a graph $G$ is {\em $S$-extendable} if there exists a cycle $C'$ in $G$ such that $V(C) \subset V(C')$ and $|C'| - |C| \in S$; $G$ is $S$-cycle extendable if every cycle of $G$ that could be $S$-extendable is $S$-extendable.  This idea was introduced and studied first in \cite{BB11}.  Clearly, being ${1}$-cycle extendable is equivalent to the original definition of being cycle extendable.  Arangno \cite{A14} conjectured that every Hamiltonian chordal graph is $\{1,2\}$-cycle extendable.  We refute this with the following theorem.

\begin{thm}
For any finite $S \subset \Z^+$, there exist infinitely many graphs which are chordal and Hamiltonian but not $S$-cycle extendable.
\end{thm}

\begin{proof}
Let $m = \max\{s \,:\, s \in S\} +1$.  Let $G_{k,m}$ denote the graph obtained from $G_k$ by replacing in $G_k$ the edge $v_kz$ with the path $v_k,v_{k+1}\cdots v_{k+m}z$ and making all new vertices complete to $\{x_1, \ldots, x_k\}$.  Let $\mathcal{H}_{k,m}$ denote the family of graphs which can be obtained from $G_{k,m}$ by pasting a clique onto each heavy edge of $G_{k,m}$, and let $H$ be any graph of $\mathcal{H}_{k,m}$.  As before, $H$ is Hamiltonian and has a cycle $C$ that spans $V(H) \setminus \{v_{k+1}, \ldots, v_{k+m}\}$.  By an argument similar to Lemma 2.5, $C$ cannot be extended to a cycle containing any of $\{v_{k+1}, \ldots, v_{k+m}\}$, and hence $C$ cannot be extended by any length in $S$.
\end{proof}

\subsection{Tree structure}\label{treestructure}

One well known characterization of chordal graphs, given by Gavril \cite{G74}, is that $G$ is chordal if and only if there exists a tree $T$ and a collection of subtrees of $T$, say $\mathcal{T}$, such that $G$ is the intersection graph of $\mathcal{T}$.  Note that $\mathcal{T}$ gives rise to an optimal tree decomposition of $G$, where the bags are precisely the maximal sets of vertices from $G$ which together form a subtree in the collection $\mathcal{T}$.  We thus call $T$ a
\textit{host tree} of $G$.
It has been shown that a Hamiltonian chordal graph $G$ is cycle extendable if 
it admits a host tree that is
a path \cite{AS06, CFGJ06} or the subdivision of a star \cite{ABS13} (that is, $G$ is a linear interval graph or a spider intersection graph, respectively).

Given a tree $T$, we call a vertex $v \in V(T)$ a {\em branch vertex} if $d_T(v) \geq 3$.  The aforementioned results on host trees show that if $G$ is a Hamiltonian chordal graph with host tree $T$, and if $T$ has at most one branch vertex, then $G$ is cycle extendable.  It easily follows that if $G$ admits a tree decomposition with host tree $T$ having at most $3$ leaves, then $G$ is cycle extendable.  Now, given a chordal graph $G$, $G$ may have many representations as the intersection graph of some host tree $T$ and collection of subtrees $\mathcal{T}$.  However, we may consider a host tree for $G$ that is minimal in terms of either the number of branch vertices or the number of leaves it contains.  The following theorems shows that the aforementioned results are almost best possible in this regard.

\begin{figure}[h!]
\begin{center}
\scalebox{1.4}{
\begin{tikzpicture}
\clip(-0.5,-0.5) rectangle (7.5,1.5);

\draw (0,0) -- (7,0);
\draw (2,0) -- (2,1);
\draw (3,0) -- (3,1);
\draw (5,0) -- (5,1);

\draw [color=purple] (-0.1,-0.1) -- (7.1,-0.1);

\draw [color=blue] (0.9,0.1) -- (6.1,0.1);
\draw [color=blue] (1.9,0.1) -- (1.9,1.1);
\draw [color=blue] (4.9,0.1) -- (4.9,1.1);

\draw [color=red] (0.9,0.2) -- (6.1,0.2);
\draw [color=red] (2.9,0.2) -- (2.9,1.1);

\draw [color=green] (0.9,-0.2) -- (2.1,-0.2);
\draw [color=green] (2.1,-0.2) -- (2.1,1.1);

\draw [color=yellow] (6.1,-0.2) -- (5.1,-0.2);
\draw [color=yellow] (5.1,-0.2) -- (5.1,1.1);

\draw [color=pink] (1.9,-0.3) -- (3.1,-0.3);
\draw [color=pink] (3.1,-0.3) -- (3.1,1.1);

\draw [color=brown] (-0.1,-0.3) -- (1.1,-0.3);

\draw [color=brown] (2.9,-0.2) -- (4.1,-0.2);

\draw [color=brown] (3.9,-0.3) -- (5.1,-0.3);

\draw [color=brown] (5.9,-0.3) -- (7.1,-0.3);

\draw [color=black] (2.2,1.1) -- (2.2,0.9);
\draw [color=black] (3.2,1.1) -- (3.2,0.9);
\draw [color=black] (5.2,1.1) -- (5.2,0.9);

\draw [color=black]  (-0.1,-0.4) -- (0.1,-0.4);
\draw [color=black]  (7.1,-0.4) -- (6.9,-0.4);

\begin{scriptsize}

\fill [color=black] (0,0) circle (2pt);
\fill [color=black] (1,0) circle (2pt);
\fill [color=black] (2,0) circle (2pt);
\fill [color=black] (3,0) circle (2pt);
\fill [color=black] (4,0) circle (2pt);
\fill [color=black] (5,0) circle (2pt);
\fill [color=black] (6,0) circle (2pt);
\fill [color=black] (7,0) circle (2pt);

\fill [color=black] (2,1) circle (2pt);
\fill [color=black] (3,1) circle (2pt);
\fill [color=black] (5,1) circle (2pt);

\end{scriptsize}
\end{tikzpicture}
}
\caption{A host tree and subtrees whose intersection graph is in $\mathcal{H}_3$}\label{tree}
\end{center}
\end{figure}
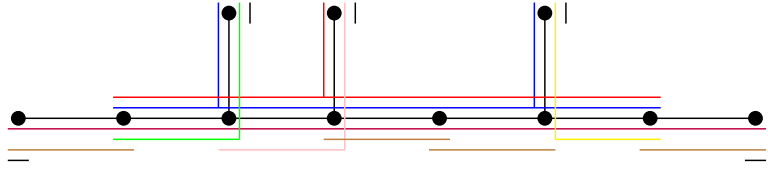


\begin{thm}
For every $t \geq 5$, there exists a counterexample to Hendry's Conjecture, say $G$, having a host tree $T$ with exactly $t-2$ branch vertices and $t$ leaves.
\end{thm}

\begin{proof}
We first show that if $H \in \mathcal{H}_k$ ($k \geq 3$), then there is a host tree for $H$ with $2k-1$ leaves and $2k-3$ branch vertices.
Recall that $H$ is constructed from $G_k = K_k \vee P_{2k+1}$, where the vertices of the $K_k$ are denoted $\{x_1, x_2, \ldots, x_k\}$ and the path $P_{2k+1}$ has vertices (in order) $\{u_1, u_2, \ldots u_k, z, v_k, \ldots, v_2, v_1\}$.

Let $P = p_1p_2\cdots p_{2k}$ be a path.  To each $p_i$ we attach a single leaf $q_i$, except for $p_{k+1}$ which remains a degree $2$ vertex.  Let the resulting tree be $T$; we claim that $T$ is a host tree for $H$.  It is easy to verify that $G_k$ is the intersection graph of the subtrees induced by the following sets of vertices in $T$:
	\begin{compactitem}[\label={}]
	\begin{multicols}{3}
	\item $u_1 = \{p_1\} \cup \{q_1\} $
	\item $u_2 = \{p_1,p_2\} \cup \{q_2\} $
	\item $\vdots $
	\item $u_k = \{p_{k-1},p_k\} \cup \{q_k\} $
	\item $z = \{p_k,p_{k+1}\} $
	\item $v_k = \{p_{k+1},p_{k+2}\} $
	\item $v_{k-1} = \{p_{k+2},p_{k+3}\} \cup \{q_{k+2}\} $
	\item $\vdots $
	\item $v_2 = \{p_{2k-1},p_{2k}\} \cup \{q_{2k-1}\} $
	\item $v_1 = \{p_{2k}\} \cup \{q_{2k}\}$
	\item $x_1 = V(P) \cup \{q_1,q_{2k}\} $
	\item $x_2 = V(P) \cup \{q_2,q_{2k-1}\} $
	\item $\vdots $
	\item $x_{k-1} = V(P) \cup \{q_{k-1},q_{k+2}\} $
	\item $x_k = V(P) \cup \{q_k\} $
	\end{multicols}
	\end{compactitem}
Figure \ref{tree} depicts this construction for $k=3$.  Now, we note that for every leaf $q_i$ of $T$, exactly two of the subtrees we have defined intersect at $q_i$ and that these are the ends of a heavy edge of $G_k$.  It is then easy to obtain $H$ as an intersection model; if a copy of $K_r$ is being pasted on to a heavy edge, we add $r-2$ distinct copies of the appropriate $q_i$ to the intersection model.  Clearly $T$ has $2k-1$ leaves and, since every $p_i$ is a branch vertex in $T$ except for $p_1,p_{k+1},$ and $p_{2k}$, $T$ has $2k-3$ branch vertices.

We now modify the construction for $\mathcal{H}_k$ ($k \geq 3$) to account for the case when $t$ is even.  From the graph $G_k$, paste a clique onto each heavy edge.  In addition, we paste a clique $X$ of order at least $k+3$ onto the clique $\{x_1, x_2, \ldots, x_k, z, v_k\}$.  We call the class of graphs which can obtained in this way $\mathcal{J}_k$ (note that such graphs are still strongly chordal).  Let $J \in \mathcal{J}_k$.  Recalling that the heavy Hamiltonian cycle shown to exist in $G_k$ contains the edge $zv_k$, it follows that $J$ is Hamiltonian.  It is easy to verify that $J$ also contains a non-extendable cycle $C$ whose vertices are $V(G_k) \setminus \{z,v_k\}$.  We have only now to construct the host tree.  To do this, we modify the construction above by adding the leaf $q_{k+1}$ to the tree, adjacent only to $p_{k+1}$, and modify the definitions of the vertices of $G_k$ as follows:
	\begin{compactitem}[\label={}]
	\begin{multicols}{3}
	\item $u_1 = \{p_1\} \cup \{q_1\} $
	\item $u_2 = \{p_1,p_2\} \cup \{q_2\} $
	\item $\vdots $
	\item $u_k = \{p_{k-1},p_k\} \cup \{q_k\} $
	\item $z = \{p_k,p_{k+1},q_{k+1}\} $
	\item $v_k = \{p_{k+1},p_{k+2},q_{k+1}\} $
	\item $v_{k-1} = \{p_{k+2},p_{k+3}\} \cup \{q_{k+2}\} $
	\item $\vdots $
	\item $v_2 = \{p_{2k-1},p_{2k}\} \cup \{q_{2k-1}\} $
	\item $v_1 = \{p_{2k}\} \cup \{q_{2k}\}$
	\item $x_1 = V(P) \cup \{q_1,q_{k+1},q_{2k}\} $
	\item $x_2 = V(P) \cup \{q_2,q_{k+1},q_{2k-1}\} $
	\item $\vdots $
	\item $x_{k-1} = V(P) \cup \{q_{k-1},q_{k+1},q_{k+2}\} $
	\item $x_k = V(P) \cup \{q_k,q_{k+1}\} $
	\end{multicols}
	\end{compactitem}
For cliques pasted onto heavy edges, we add copies of the appropriate $q_i$'s ($i \neq k+1$) as before.  For each vertex in $X \setminus \{x_1, x_2, \ldots, x_k, z, v_k\}$, we use $q_{k+1}$ to represent it in the intersection model.  The resulting tree has $2k$ leaves and $2k-2$ branch vertices.
\end{proof}

We pose the following natural problem

\begin{ques}
Does there exist a counterexample to Hendry's Conjecture $G$ admitting a tree decomposition with a host tree having four leaves, but not admitting a tree decomposition with host tree having three leaves?  If not, is there a counterexample admitting a tree decomposition with a host tree having two branch vertices, but not admitting a tree decomposition with host tree having one branch vertex?
\end{ques}

In addition, we note that every counterexample that we have constructed has a host tree whose maximum degree is $3$.  As mentioned, tree decompositions are certainly not unique; ${H}_{3}$ has another decomposition tree with a vertex of degree $4$, which can be obtained by deleting $p_{3} q_{3}$ and adding $p_{2} q_{3}$.  However, one may consider among all tree decompositions of a graph those host trees having smallest possible maximum degree $\Delta(T)$. If $\min\{\Delta(T) \,:\, (T,\mathcal{T}) \textrm{ is a tree decomposition for $G$}\} = 2$, then $G$ is a linear interval graph and thus satisfies Hendry's Conjecture.  If $\min\{\Delta(T) \,:\, (T,\mathcal{T}) \textrm{ is a tree decomposition for $G$}\} = 3$, then $G$ may satisfy Hendry's Conjecture (if it is a spider intersection graph) or not (if $G \in \mathcal{H}_k$).  This leads us to the following question:

\begin{ques}
Suppose $G$ is a Hamiltonian chordal graph.  Is $G$ cycle extendable if $\min\{\Delta(T) \,:\, (T,\mathcal{T}) \textrm{ is a tree decomposition for $G$}\} = 4$?
\end{ques}








\section{Concluding remarks}\label{conclusion}

We considered an extremal problem related to Hendry's Conjecture, motivated by the following theorem due to Hendry (which is a generalization of classic results of Ore \cite{O61} and Bondy \cite{B72}):

\begin{thm}[Hendry \cite{H90}]\label{dense}
Let $G$ be a graph with $|V(G)| = n$.  If $|E(G)| \geq {n-1 \choose 2} + 1$, then either $G$ is fully cycle extendable or $G$ is isomorphic to one of the following graphs: 
	\begin{compactitem}
	\item $K_1 \vee \left( K_1 \cup K_{n-2} \right)$
	\item $K_2 \vee \overline{K_3}$
	\item $\overline{K_2} \vee \left( K_1 \cup K_{n-3} \right)$
	\end{compactitem}
\end{thm}

One can further consider this extremal problem by imposing extra conditions on the graph.  Let $f_E(n)$ denote the minimum number of edges required to guarantee that an $n$-vertex Hamiltonian graph is cycle extendable, and let $g_E(n)$ denote the minimum number of edges required to guarantee that an $n$-vertex Hamiltonian chordal graph is cycle extendable.

\begin{prob}
Determine upper and lower bounds on $f_E(n)$ and $g_E(n)$.
\end{prob}

It is easy to see that there exist counterexamples to Hendry's Conjecture with high density.  Let $D_n \in \mathcal{H}_3$ be the $n$-vertex graph obtained from $G_3$ where the five cliques pasted consist of four copies of $K_3$ and one copy of $K_{n-12}$.  Since $|E(D_n)| = {n-12 \choose 2} + 37$, we have that $\frac{n^2 - 25n + 230}{2} < f_E(n) \leq g_E(n)$, but this is almost certainly not the best possible lower bound.

\acknowledgements

This work was completed while the first author was affiliated with the Department d'informatique et de recherche op\'erationnelle, Universit\'e de Montr\'eal, Canada and the third author was affiliated with the Department of Computer Engineering, Sharif University of Technology, Tehran, Iran.  We thank the anonymous referees for the feedback, particularly regarding the presentation of Section \ref{treestructure}.

\bibliographystyle{siamplain}

\end{document}